\documentclass[11pt]{article}

\usepackage{amsmath,amsfonts,amssymb}
\usepackage{amsthm}
\usepackage{graphicx}
\usepackage{caption,subcaption}
\usepackage{booktabs}
\usepackage{xfrac}
\usepackage{algorithm,algorithmic}
\usepackage{hyperref}
\usepackage[margin=1in]{geometry}
\graphicspath{{figs/}}
\newcommand{\mbf}{\mathbf}             % bold
\newcommand{\mbb}{\mathbb}				   % blackboard bold
\newcommand{\tr}{^{\textnormal{T}}}    % matrix/vector transpose
\newcommand{\abs}[1]{\left| #1 \right|}% absolute value
\newcommand{\norm}[1]{\left\| #1 \right\|}% norm
\newcommand{\st}{\textnormal{s.t.}\:}  % such that
\newcommand{\alev}{a.e.\;}	            % almost everywhere
\renewcommand{\hat}{\widehat}
\renewcommand{\tilde}{\widetilde}

\newcommand{\median}{\textnormal{median}}

\newcommand{\smallsum}{{\textstyle{\sum}}}

\renewcommand{\a}{\mathbf{a}}

\renewcommand{\d}{\mathbf{d}}

\newcommand{\f}{\mathbf{f}}
\renewcommand{\l}{\boldsymbol{\ell}}

\renewcommand{\u}{\mathbf{u}}
\renewcommand{\v}{\mathbf{v}}
\newcommand{\w}{\mathbf{w}}
\newcommand{\x}{\mathbf{x}}
\newcommand{\y}{\mathbf{y}}
\newcommand{\z}{\mathbf{z}}

\newtheorem{theorem}{Theorem}
\newtheorem{corollary}[theorem]{Corollary}

\theoremstyle{definition}

\newtheorem{definition}{Definition}

\begin{document}

\title{Enhancing interval observers for state estimation using constraints
\thanks{This work was funded by Novartis Pharmaceuticals as part of the Novartis-MIT Center for Continuous Manufacturing.}
}
\author{Stuart M. Harwood%
\and Paul I. Barton%
\thanks{Process Systems Engineering Laboratory, Massachusetts Institute of Technology, Cambridge, MA 02139 (\texttt{pib@mit.edu})}
}
\maketitle

%\runningheads{S.M. Harwood and P.I. Barton}{}

\begin{abstract}
This work considers the problem of calculating an interval-valued state estimate for a nonlinear system subject to bounded inputs and measurement errors.
Such state estimators are often called interval observers.
Interval observers can be constructed using methods from reachability theory.
Recent advances in the construction of interval enclosures of reachable sets for nonlinear systems inspire the present work.
These advances can incorporate constraints on the states to produce tighter interval enclosures.
When applied to the state estimation problem, bounded-error measurements may be used as state constraints in these new theories.
The result is a method that is easily implementable and which generally produces better, tighter interval state estimates.
Furthermore, a novel linear programming-based method is proposed for calculating the observer gain, which must be tuned in practice.
In contrast with previous approaches, this method does not rely on special system structure.
The new approaches are demonstrated with numerical examples.
\end{abstract}
%\keywords{%
%state estimation;
%interval observers;
%nonlinear dynamics;
%}

%%%%%%%%%%%%%%%%%%%%%%%%%%%%%%%%%%%%%%%%%%%%%%%%%%%%%%%%%%%%%%%%%%%%%%%%%%%%%%%%
%%%%%%%%%%%%%%%%%%%%%%%%%%%%%%%%%%%%%%%%%%%%%%%%%%%%%%%%%%%%%%%%%%%%%%%%%%%%%%%%
\section{Introduction}
\label{sec:intro}

This work considers the problem of bounded-error state estimation for nonlinear dynamic systems with continuous-time measurements.
The goal of state estimation is to determine or estimate the internal state of a real system.
While only certain states or functions of the states can be measured directly, a mathematical model of the system is available.
When the system is a dynamic one, a history of measurements is typically available.
Estimating the state using a dynamic model and measurements goes back to the Kalman filter \cite{kalman60}, where a more statistical estimate of the state is obtained.
More recently, work has focused on estimates in the form of guaranteed bounds on the state in the presence of bounded uncertainties and measurement errors \cite{efimovEA13,jaulin02,meslemEA10,moisanEA07,thabetEA14}.
This work continues in this direction.
%Meanwhile, \cite{kiefferEA11,kiefferEA06} deal with interval bounds on the reachable set in the context of parameter estimation problems.

Interval observers are a type of state estimator most closely related to the present work.
The approach to constructing interval observers taken in \cite{meslemEA10,meslemEA11,moisanEA05,moisanEA07} is similar to the theory that we will use in the proposed approach.
Underpinning these methods are theorems for general estimation of reachable sets.
These theories are in the vein of comparison theorems involving differential inequalities.
Such theorems have a long history, going back to ``M\"{u}ller's theorem'' \cite{walter},
which subsequently was generalized to control systems in \cite{harrison77}.
Recent work from \cite{harwoodEA15,harwoodEA16,harwoodEA17,scottEA13} develops these kinds of comparison theorems further.
In particular, they incorporate state constraints into the bounding theory in a fundamental way;
this information typically produces a tighter estimate of the reachable set without significant extra computational effort.

While the work in \cite{meslemEA10,meslemEA11,moisanEA05,moisanEA07} relies on the classic M\"{u}ller's theorem or the theory of cooperative systems, the present work depends on the more recent method for reachability analysis in \cite{harwoodEA17}.
When this theory is applied to state estimation, bounded error measurements take the role of state constraints, and tighter state estimates result.

Related literature includes work on general reachability analysis with constraints, which has been addressed in \cite{kurzhanskiEA06}.
In that work, the focus is on linear systems and ellipsoidal enclosures of the reachable sets.
Related work deals with control problems with state constraints \cite{bokanowskiEA10,kurzhanskiEA06_optimization}, in which  the theoretical basis of the approaches is on the Hamilton-Jacobi-Bellman partial differential equation (PDE);
the authors of \cite{kurzhanskiEA06_optimization} note that the solution of such a PDE is in general complicated for nonlinear systems.

%Another contribution of the present work is the application of the theory in \cite{scottEA13,harwoodEA16} to  this ``set-membership'' approach to state estimation;
%the previous work in \cite{scottEA13,harwoodEA15,harwoodEA16} considered purely reachability analysis and have not explored the idea of using measurement information to address the state estimation problem.

The contributions and structure of this work are as follows.
Section~\ref{sec:prelim} discusses notation and formally states the problem of interest.
Section~\ref{sec:se} briefly discusses some of the ideas behind interval state estimates, and then proceeds to detail the proposed method.
Again, the novelty of the proposed state estimation method is the incorporation of measurements as constraints and applying recent advances in reachability analysis.
Further, a new procedure for calculating the ``gain'' of the observer system is discussed.
In contrast with previous work, this method does not rely on any special system structure such as cooperativity.
Section~\ref{sec:examples} demonstrates a numerical implementation of the new method for state estimation.
Comparing it with related methods, it is shown that better, tighter state estimates can be obtained with the new method.
Section~\ref{sec:conc} concludes with some final remarks.

%%%%%%%%%%%%%%%%%%%%%%%%%%%%%%%%%%%%%%%%%%%%%%%%%%%%%%%%%%%%%%%%%%%%%%%%%%%%%%%%
%%%%%%%%%%%%%%%%%%%%%%%%%%%%%%%%%%%%%%%%%%%%%%%%%%%%%%%%%%%%%%%%%%%%%%%%%%%%%%%%
\section{Problem statement and preliminaries}
\label{sec:prelim}

%%%%%%%%%%%%%%%%%%%%%%%%%%%%%%%%%%%%%%%%%%%%%%%%%%%%%%%%%%%%%%%%%%%%%%%%%%%%%%%%
\subsection{Notation}
First, notation that is used in this work includes lowercase bold letters for vectors and uppercase bold letters for matrices.
The transposes of a vector $\v$ and matrix $\mbf{M}$ are denoted $\v \tr$ and $\mbf{M} \tr$, respectively.
The exception is that $\mbf{0}$ may denote either a matrix or vector of zeros, but it should be clear from context what the appropriate dimensions are.
Similarly, $\mbf{1}$ denotes a vector of ones whose dimension should be clear from context.
%Similarly, $\mbf{I}$ will denote the identity matrix where the dimension should be clear from context.
The $j^{th}$ component of a vector $\mbf{v}$ will be denoted ${v}_j$.
For a matrix $\mbf{M} \in \mbb{R}^{p \times n}$, the notation 
$\mbf{M} = [ \mbf{m}_i \tr ]$ 
may be used to emphasize that the $i^{th}$ row of $\mbf{M}$ is $\mbf{m}_i$, for $i \in \{1, \dots, p\}$.
Similarly, $\mbf{M} = [m_{i,j}]$ emphasizes that the element in the $i^{th}$ row and $j^{th}$ column is $m_{i,j}$.
Inequalities between vectors hold componentwise.
For $\v$, $\w \in \mbb{R}^n$ such that $\v \le \w$, $[\v, \w]$ denotes a nonempty interval in $\mbb{R}^n$.
For a set $Y \subset \mbb{R}^n$, denote the set of nonempty interval subsets of $Y$ by $\mbb{I}Y$.
%A polyhedron is any subset of $\mbb{R}^n$ that can be expressed as $\{ \z \in \mbb{R}^n : \mbf{M} \z \le \d \}$, for some matrix $\mbf{M} \in \mbb{R}^{m \times n}$ and vector $\d \in \mbb{R}^m$ (i.e. it is the intersection of a finite number of closed halfspaces).
%Consequently, it is clear that polyhedra are always closed, convex sets.
%
The equivalence of norms on $\mbb{R}^n$ is used often;
when a statement or result holds for any choice of norm, it is denoted $\norm{\cdot}$.
In some cases, it is useful to reference a specific norm, in which case it is subscripted;
for instance, $\norm{ \cdot }_1$ denotes the 1-norm.
The dual norm of a norm $\norm{\cdot}$ is denoted $\norm{\cdot}_*$.
%For $n \in \mbb{N}$ and a set $T \subset \mbb{R}$, denote the Lebesgue space
%$L^1(T,\mbb{R}^n) \equiv \{ (\v : T \to \mbb{R}^n): \int_T \abs{v_i} < +\infty, \forall i \}$.
%That is, $\v \in L^1(T,\mbb{R}^n)$ if each component of $\v$ is in $L^1(T) \equiv L^1(T,\mbb{R})$.
%
%For sets $X$, $Y$, a mapping $S$ from $X$ to the set of subsets of $Y$ is denoted $S : X \svm Y$.
In a metric space, a neighborhood of a point $x$ is denoted $N(x)$ and refers to an open ball centered at $x$ with some nonzero radius.

%%%%%%%%%%%%%%%%%%%%%%%%%%%%%%%%%%%%%%%%%%%%%%%%%%%%%%%%%%%%%%%%%%%%%%%%%%%%%%%%
\subsection{Problem statement}
\label{sec:prob}
The problem of interest is one in estimating the state of a continuous-time dynamic system;
in particular, we seek a rigorous enclosure of the states or a bounded error estimate.
The dynamic model of the system is defined by:
for positive integers $n_x,n_y,n_u$,
nonempty interval $T = [t_0,t_f] \subset \mbb{R}$,
$D_x \subset \mbb{R}^{n_x}$, 
$D_u \subset \mbb{R}^{n_u}$,
$\mbf{f} : T \times D_u \times D_x \to \mbb{R}^{n_x}$, and
$\mbf{C} \in \mbb{R}^{n_y \times n_x}$.
The system obeys the following equations:
\begin{subequations}
\label{eq:system}
\begin{align}
\label{eq:ivp}
\dot{\x}(t) &= \mbf{f}(t, \u(t), \x(t)), \quad \alev t \in T, \\
\label{eq:meas}
\y(t) &= \mbf{C} \x(t) + \v(t), \quad \forall t \in T,
\end{align}
\end{subequations}
where 
$\x(t)$ is the state of the system, 
$\y(t)$ is the measurement or output of the system, 
$\u(t)$ is a disturbance to the system, and 
$\v(t)$ is the measurement noise.

In this work, we exclusively consider absolutely continuous solutions $\x$ of the IVP in ODEs \eqref{eq:ivp}.
Further, we assume that it is known that
the initial conditions satisfy $\x(t_0) \in X_0 \in \mbb{I}D_x$,
$\v$ is measurable and $\v(t) \in V(t) \in \mbb{I}\mbb{R}^{n_y}$ for all $t \in T$,
$\u$ is measurable and $\u(t) \in U(t) \in \mbb{I}D_u$ for all $t \in T$,
where $V(t) \subset K_v$ and $U(t) \subset K_u$, for all $t \in T$, for compact $K_v$ and $K_u$.
In other words, interval enclosures of the possible initial conditions, measurement noise, and disturbance values are known.
A fundamental assumption on the dynamics $\f$ is a kind of local Lipschitz condition:
%\begin{assumption}
%\label{assm:lip}
For any $\z \in D_x$, there exists a neighborhood $N(\z)$ and $\alpha \in L^1(T)$ such that for almost every $t \in T$ and every $\w \in U(t)$
\[
\norm{ \f(t,\w,\z_1) - \f(t,\w,\z_2) } \le \alpha(t) \norm{ \z_1 - \z_2 },
\]
for every $\z_1,\z_2 \in N(\z) \cap D_x$.
%\end{assumption}

The challenge, then is to construct an interval $[\x^L(t_f),\x^U(t_f)]$ such that
\[
\x(t_f) \in [\x^L(t_f),\x^U(t_f)]
\]
for any possible realization of noise $\v$, disturbance $\u$, and initial condition.

For future reference, since $V$ and $U$ are interval-valued, we can write
$V(t) = [\v^L(t),\v^U(t)]$ and $U(t) = [\u^L(t),\u^U(t)]$ for appropriate functions.

%%%%%%%%%%%%%%%%%%%%%%%%%%%%%%%%%%%%%%%%%%%%%%%%%%%%%%%%%%%%%%%%%%%%%%%%%%%%%%%%
%%%%%%%%%%%%%%%%%%%%%%%%%%%%%%%%%%%%%%%%%%%%%%%%%%%%%%%%%%%%%%%%%%%%%%%%%%%%%%%%
\section{State estimation}
\label{sec:se}

\subsection{Motivation}

Previous work on bounded-error state estimation includes interval observers;
we refer to \cite{efimovEA16} for a recent review of such work and some related topics.
The essence of this type of method is easily understood for a linear system.
Consider 
\begin{align}
\dot\x(t) &= \mbf{A}\x(t), \\
\y(t) &= \mbf{C}\x(t) + \v(t).
\end{align}
Then for any matrix $\mbf{L} \in \mbb{R}^{n_x \times n_y}$, we can write
\[
\dot\x(t) = (\mbf{A} - \mbf{LC}) \x(t) + \mbf{LC}\x(t),
\]
and subsequently
\[
\dot\x(t) = (\mbf{A} - \mbf{LC}) \x(t) - \mbf{L}\v(t) + \mbf{L}\y(t).
\]
The theory of interval observers often focuses on the case that the ``gain matrix'' $\mbf{L}$ is chosen so that 
$\mbf{A} - \mbf{LC}$ has positive off-diagonal components, and
that $\norm{\w}_{\infty} \le v^N$, for all $\w \in V(t)$, for all $t$
(so that 
$\mbf{L}\v(t) \le \norm{\mbf{L}\v(t)}_{\infty} \mbf{1} \le \norm{\mbf{L}}_{\infty}v^N \mbf{1}$,
implying
$-\norm{\mbf{L}}_{\infty}v^N \mbf{1} \le \mbf{L}\v(t) \le \norm{\mbf{L}}_{\infty}v^N\mbf{1}$).
In this case one can write
\begin{align*}
\dot\x^L(t) &= (\mbf{A} - \mbf{LC}) \x^L(t) - \norm{\mbf{L}}_{\infty}v^N \mbf{1} + \mbf{L}\y(t), \\
\dot\x^U(t) &= (\mbf{A} - \mbf{LC}) \x^U(t) + \norm{\mbf{L}}_{\infty}v^N \mbf{1} + \mbf{L}\y(t).
\end{align*}
Then assuming $X_0 \subset [\x^L(t_0),\x^U(t_0)]$, the theory of monotone dynamic systems ensures that for any initial condition and realization of measurement noise, $\x(t) \in [\x^L(t),\x^U(t)]$, for all $t \in T$.

Our approach is to extend this idea with more general theorems for reachability analysis based on the theory of differential inequalities,  and in particular, to leverage the recent developments that take advantage of constraint information to tighten the estimates of the reachable sets.
To this end, we interpret the state estimation problem as a reachability problem for an IVP in \emph{constrained} ODEs.
This has been considered most recently in \cite{harwoodEA17}.
Similar techniques from \cite{harwoodEA15,harwoodEA16,scottEA13} also apply;
those articles use ``a priori enclosures'' of the reachable set, which are in effect constraints, although they do not necessarily exclude any solutions.
As demonstrated in  \cite{harwoodEA17}, these techniques can be applied to constrained systems, if one is only interested in solutions that also satisfy the constraints.

Thus, the proposed approach is to apply reachability analysis to the system
\begin{align*}
\dot{\x}(t) &= \mbf{f}(t, \u(t), \x(t)) - \mbf{LC} \x(t) - \mbf{L}\v(t) + \mbf{L}\y(t), 
   \quad \alev t \in T, \\
\x(t) &\in X_c(t) \quad \forall t \in T,
\end{align*}
for some choice of matrix $\mbf{L}$, and 
constraint mapping $X_c$ defined by 
$X_c(t) = \{ \z: \y(t) - \v^U(t) \le \mbf{C}\z \le \y(t) - \v^L(t)\}$.

Recent developments in the design of interval observers from \cite{efimovEA13,thabetEA14} allow the gain matrix $\mbf{L}$ to vary, depending on time and measurements.
There is no significant hurdle to incorporating such a generalization in the method that follows;
we choose not to do so in order to keep notation simple and to highlight the essence of our contribution.

%%%%%%%%%%%%%%%%%%%%%%%%%%%%%%%%%%%%%%%%%%%%%%%%%%%%%%%%%%%%%%%%%%%%%%%%%%%%%%%%
\subsection{Proposed method}
\label{sec:method}

To state the method, we must introduce a few definitions and operations.
Central to the method for estimating reachable sets that we will use, from \cite{scottEA13}, is the ability to overestimate or bound the dynamics. 
We will do so with an inclusion function constructed using principles from interval arithmetic.
See \cite{moore_etal} for an introduction to interval arithmetic and inclusion functions.
As in the discussion above, the dynamics we care about are for a modified system depending on the gain matrix $\mbf{L}$.
%The modified dynamics and their inclusion function are now defined.
%
\begin{definition}
\label{defn:inclusionfun}
For $\mbf{L} \in \mbb{R}^{n_x \times n_y}$,
define 
\begin{equation}
\notag
\hat\f : (t, \u, \z, \v; \mbf{L}) \mapsto \f(t,\u,\z) - \mbf{LC} \z - \mbf{L}\v,
\end{equation}
and let $\hat\f^L$ and $\hat\f^U$ be the endpoints of an 
inclusion monotonic interval extension of
$\hat\f(\cdot,\cdot,\cdot,\cdot; \mbf{L})$.
\end{definition}
The critical property of an inclusion monotonic interval extension is that it is an inclusion function;
that is
\[
\hat\f^L([t,t],U', Z', V'; \mbf{L}) \le
\hat\f(t,      \u, \z, \v; \mbf{L}) \le
\hat\f^U([t,t],U', Z', V'; \mbf{L})
\]
for any 
$U' \in \mbb{I}D_u$,
$Z' \in \mbb{I}D_x$,
$V' \in \mbb{IR}^{n_y}$,
and $(t,\u,\z,\v) \in T \times U' \times Z' \times V'$.
Inclusion functions for a broad class of functions can be automatically and cheaply evaluated with a number of numerical libraries implementing interval arithmetic,
for instance, MC++ \cite{mc++}, INTLAB \cite{rump99}, or PROFIL \cite{knuppel94}.
The proposed method relies on these automatically computable inclusion functions, which reduces the amount of analysis that must be performed compared to other work.
For instance,  the work in \cite{meslemEA10,meslemEA11} requires analysis of the dynamics to derive a hybrid automata which is then used to construct the state estimate.
%\cite{efimovEA13,thabetEA14} is all about coordinate transformations to get Metzler matrices

Since it will be useful later, we mention that the \emph{natural interval extension} of a function is an inclusion monotonic interval extension.
For a linear function $g : \z \mapsto \a\tr\z$, the natural interval extension $[g^L, g^U]$ is computed by applying the rules of interval arithmetic:
\begin{align*}
[v_i^L, v_i^U] &= 
\begin{cases}
[a_i z_i^L, a_i z_i^U] & \text{if } a_i\ge 0, \\
[a_i z_i^U, a_i z_i^L] & \text{else},
\end{cases}\\
[g^L, g^U] &= [\smallsum_i v_i^L, \smallsum_i v_i^U ].
\end{align*}

We will also require the following definition.
If $\v \le \w$, then $B_i^L(\v,\w)$ returns the $i^{th}$ lower face of the interval $[\v,\w]$.
If $[\v,\w]$ is empty, then a nonempty interval $[\hat\v,\hat\w]$ is constructed and the faces of that interval are returned.
\begin{definition}
Define for each $i \in \{ 1,2,\dots,n_x \}$
\begin{align*}
B_i^L &: (\v,\w) \mapsto \{ \z \in [\hat\v,\hat\w] : z_i = \hat{v}_i \}, \\ 
B_i^U &: (\v,\w) \mapsto \{ \z \in [\hat\v,\hat\w] : z_i = \hat{w}_i \},
\end{align*}
where $\hat\v$, $\hat\w$ are given componentwise by
$\hat{v}_j = \min\{ v_j, (v_j + w_j)/2 \}$ and
$\hat{w}_j = \max\{ w_j, (v_j + w_j)/2 \}$.
\end{definition}
%
%This ``flattening'' operation is critical to bounding theorems based on differential inequalities.

The operation defined in Algorithm~\ref{alg:tighten} is required.
Originally from Definition 4 in \cite{scottEA13}, this algorithm defines the operation $I_t$, which tightens an interval $[\v,\w]$ by excluding points which cannot satisfy a given set of linear constraints $\mbf{M} \z \le \d$.
Specifically, the discussion in \S5.2 of \cite{scottEA13} establishes that the tightened interval $I_t([\v,\w], \d; \mbf{M})$ satisfies
\[
\{ \z \in [\v,\w] : \mbf{M} \z \le \d \} \subset I_t([\v,\w], \d; \mbf{M}) \subset [\v,\w].
%[\v,\w] \supset I_t([\v,\w], \d; \mbf{M}) \supset \{ \z \in [\v,\w] : \mbf{M} \z \le \d \}.
\]

\begin{algorithm}
\caption{Definition of the interval-tightening operator $I_t$}
\label{alg:tighten}
\begin{algorithmic}
\REQUIRE 
positive integers $n_m,n$,
$\mbf{M} = [m_{i,j}] \in \mbb{R}^{n_m \times n}$, 
$\d \in \mbb{R}^{n_m}$, 
$(\v,\w) \in \mbb{R}^{n} \times \mbb{R}^{n}$, $\v \le \w$
\STATE 
	$(\hat{\v}, \hat{\w}) \gets (\v,\w)$
\FOR{ $i \in \{1, \dots, n_m\}$ }
\FOR{ $j \in \{1, \dots, n\}$ }
	\IF{ $m_{i,j} \neq 0$ }
		\STATE{
		$\gamma \gets \median{} \left\{ \hat{v}_j, \hat{w}_j, \sfrac{1}{m_{i,j}} \big( d_i + \sum_{k \neq j} \max\{ -m_{i,k} \hat{v}_k, -m_{i,k} \hat{w}_k \}  \big) \right\}$
		\IF{ $m_{i,j} > 0$ }
			\STATE{ $\hat{w}_j \gets \gamma$ }
		\ENDIF
		\IF{ $m_{i,j} < 0$ }
			\STATE{ $\hat{v}_j \gets \gamma$ }
		\ENDIF
		}
	\ENDIF
\ENDFOR
\ENDFOR
\RETURN $I_t([\v,\w], \d; \mbf{M}) \gets [\hat{\v}, \hat{\w}]$
\end{algorithmic}
\end{algorithm}

With the definition of the interval tightening operator, we can define an operation specific to our problem which tightens an interval based on the constraints 
$\mbf{C}\z \le \y(t) - \v^L(t)$,
$\mbf{C}\z \ge \y(t) - \v^U(t)$.

\begin{definition}
Let 
\[
\d : (t,\y) \mapsto \begin{bmatrix} \y - \v^L(t) \\ -\y + \v^U(t) \end{bmatrix}, \qquad
\mbf{M} = \begin{bmatrix} \mbf{C} \\ -\mbf{C} \end{bmatrix}.
\]
Define 
\[
I_c : ([\v,\w],t,\y) \mapsto I_t([\v,\w],\d(t,\y);\mbf{M}).
\]
\end{definition}

We can now state the specific method for constructing a state estimator.
The main difference between this method and a classic comparison theorem for reachability (see \cite{harrison77}) is the application of the $I_c$ operator.
While the classic method would overstimate, for instance, $\hat{f}_i$ on the set 
$\{ \z \in [\x^L(t), \x^U(t)] : z_i = x_i^U(t) \}$
(the $i^{th}$ upper face of the interval),
the proposed method applies $I_c$ to that set, and overestimates $\hat{f}_i$ on the resulting (no larger) interval.
\begin{theorem}
\label{thm:method}
For any $\mbf{L} \in \mbb{R}^{n_x \times n_y}$, 
let $(\x^L,\x^U)$ be absolutely continuous and satisfy
\begin{align}
\notag
\dot{x}_i^L(t) &= \hat{f}_i^L([t,t], U(t), I_c(B_i^L(\x^L(t),\x^U(t)),t,\y(t)), V(t); \mbf{L}) + \mbf{L}\y(t), \quad \alev t \in T, \\
\notag
\dot{x}_i^U(t) &= \hat{f}_i^U([t,t], U(t), I_c(B_i^U(\x^L(t),\x^U(t)),t,\y(t)), V(t); \mbf{L}) + \mbf{L}\y(t), \quad \alev t \in T,
\end{align}
for each $i \in \{1,2,\dots,n_x\}$, and with initial conditions
$[\x^L(t_0),\x^U(t_0)] = X_0$.
Then $\x(t) \in [\x^L(t),\x^U(t)]$ for all $t \in T$;
that is, $[\x^L(t),\x^U(t)]$ is an interval estimate of $\x(t)$ for all $t$.
\end{theorem}
\begin{proof}
As mentioned earlier, the claim follows by applying an established method for estimating reachable sets for the system
\[
\dot{\x}(t) = \mbf{f}(t, \u(t), \x(t)) - \mbf{LC} \x(t) - \mbf{L}\v(t) + \mbf{L}\y(t), \quad \alev t \in T,
\]
where $\x(t_0) \in X_0$, $(\u(t),\v(t)) \in U(t) \times V(t)$.
The specific method is from \cite[\S5.2]{scottEA13}.
With the assumptions on the problem setting in \S\ref{sec:prob} and the definitions of the dynamics defining $(\x^L,\x^U)$, all the assumptions and hypotheses of the method from \cite{scottEA13} are satisfied, and the conclusion holds that $\x(t) \in [\x^L(t),\x^U(t)]$ for any solution $\x$ satisfying
$\y(t) - \v^U(t) \le \mbf{C}\x(t) \le \y(t) - \v^L(t)$.
\end{proof}

We note that Theorem~\ref{thm:method} does not guarantee the \emph{existence} of $\x^L$ and $\x^U$;
it provides a construction that, if successful, yields an interval estimate.
Conditions that we have not stated explicitly,
such as continuity of $\hat\f^L$ and $\hat\f^U$ and measurability of $\v^L$, $\v^U$, $\u^L$, $\u^U$ 
(defining the set-valued mappings $V$ and $U$),
may be required to apply standard existence results for the solutions of IVPs in ODEs.
%see Theorem 1.10 of Chapter II of \cite{mattheij_etal} or Theorem 2 in Section~1 of \cite{filippov}.
Perhaps more important are the typically stronger conditions that guarantee the applicability of numerical methods for the solution of IVPs in ODEs;
see, for instance, Theorem 1.1 of Section~1.4 of \cite{lambert}.
For example, in the numerical examples that follow, we use measurements $\y$ that are a continuous function on $T$.

\subsection*{Comparisons}

In \S\ref{sec:examples}, we will compare the proposed method established above against its variants and others from the literature.
The method from Theorem~\ref{thm:method} with $\mbf{L} = \mbf{0}$ is called the ``No Measurements'' method;
while it still uses measurement information as constraints, it does not use the measurement values $\y(t)$ directly.
Meanwhile, the ``No Constraints'' method is a variant of the proposed method that does not use the constraint information;
that is, the endpoints of the interval estimate satisfy the differential equations
\begin{align}
\notag
\dot{x}_i^L(t) &= \hat{f}_i^L([t,t], U(t), B_i^L(\x^L(t),\x^U(t)), V(t); \mbf{L}) + \mbf{L}\y(t), \quad \forall i,\\
\notag
\dot{x}_i^U(t) &= \hat{f}_i^U([t,t], U(t), B_i^U(\x^L(t),\x^U(t)), V(t); \mbf{L}) + \mbf{L}\y(t), \quad \forall i.
\end{align}
Note that application of the operation $I_c$ has been omitted.

%%%%%%%%%%%%%%%%%%%%%%%%%%%%%%%%%%%%%%%%%%%%%%%%%%%%%%%%%%%%%%%%%%%%%%%%%%%%%%%%
\subsection{Linear stability analysis and automatic gain calculation}

Methods from the literature for calculating the gain matrix $\mbf{L}$ often focus on the case that the system is linear, e.g. $\dot{\x} = \mbf{A}\x$.
The methods often involve finding a gain matrix $\mbf{L}$ so that $\mbf{A} - \mbf{LC}$ has nonnegative off-diagonal components;
a matrix with this last property is often called a Metzler matrix.
Consequently, the theory of cooperative or monotone systems can be applied to derive interval estimates.

In the present section, our goal is to state a method for calculating the gain matrix that does not depend on the system being cooperative.
We begin with a fundamental result which will motivate the method;
we show that the No Constraints method with a gain matrix satisfying certain constraints will produce an asymptotically exact interval estimate for an idealized linear system.

\begin{theorem}
\label{thm:LinearStability}
Suppose that the system of interest has a homogeneous, linear time invariant form, with exact measurements:
\begin{align}
\notag
\dot{\x}(t) &= \mbf{A}\x(t),\quad \alev t, \\
\notag
\y(t) &= \mbf{C}\x(t), \quad \forall t.
\end{align}
Consider the No Constraints method for calculating an interval state estimate;
assume that the natural interval extension is used to calculate the inclusion monotonic interval extension function required in Definition~\ref{defn:inclusionfun}.
Let the columns of $\mbf{C}$ be $\mbf{C}_i$.
Then for any gain matrix $\mbf{L} = [\l_i\tr]$ which satisfies
\begin{equation}
\label{eq:GainConstraints}
{a}_{i,i} - \l_i\tr \mbf{C}_i + \sum_{j \ne i} \abs{ {a}_{i,j} - \l_i\tr \mbf{C}_j } < 0, \quad \forall i,
\end{equation}
the interval state estimate $[\x^L,\x^U]$ resulting from the No Constraints method is asymptotically exact;
i.e.
$\x^U(t) - \x^L(t) \to \mbf{0}$ as $t \to +\infty$.
\end{theorem}
\begin{proof}
For any $\mbf{L}$,
let $\mbf{M} = \mbf{A} - \mbf{LC}$
(clearly $\mbf{M}$ depends on $\mbf{L}$ but we will suppress this dependence in the notation).
Let $\mbf{M} = [\mbf{m}_i\tr]$ 
so that the state estimates $(\x^L,\x^U)$ resulting from the No Constraints method satisfy the following differential equations for each $i$ and almost every $t$:
\begin{align}
\dot{x}_i^L(t) &= \mbf{m}_i\tr \mbf{B}_i^L (\x^L(t),\x^U(t)) + \mbf{L}\y(t),\\
\dot{x}_i^U(t) &= \mbf{m}_i\tr \mbf{B}_i^U (\x^L(t),\x^U(t)) + \mbf{L}\y(t).
\end{align}
Here, $\mbf{B}_i^L : \mbb{R}^{2n} \to \mbb{R}^{n}$ is a linear mapping whose effect is the following:
if
$
\z^L = \mbf{B}_i^L (\x^L,\x^U),
$
then
\[
z_j^L = 
\begin{cases}
x_i^L & \text{if }   j = i, \\
x_j^L & \text{if }   j \ne i \text{ and } m_{i,j} \ge 0, \\
x_j^U & \text{else }(j \ne i \text{ and } m_{i,j} <   0).
\end{cases}
\]
Simply, $\mbf{m}_i\tr \mbf{B}_i^L (\x^L(t),\x^U(t))$ is exactly the value of the lower bound of the natural interval extension of $\x \mapsto \mbf{m}_i\tr \x$ on the set $B_i^L(\x^L(t), \x^U(t))$, as we would expect of the No Constraints method.
Similarly, if
$
\z^U = \mbf{B}_i^U (\x^L,\x^U),
$
then
\[
z_j^U = 
\begin{cases}
x_i^U & \text{if }   j = i, \\
x_j^U & \text{if }   j \ne i \text{ and } m_{i,j} \ge 0, \\
x_j^L & \text{else }(j \ne i \text{ and } m_{i,j} <   0),
\end{cases}
\]
and again, $\mbf{m}_i\tr \mbf{B}_i^U (\x^L(t),\x^U(t))$ equals the upper bound of the natural interval extension of $\x \mapsto \mbf{m}_i\tr \x$ on the set $B_i^U(\x^L(t), \x^U(t))$.

Then we have that for each $i$
\[
\dot{x}_i^U(t) - \dot{x}_i^L(t) = m_{i,i}(x_i^U(t) - x_i^L(t)) + \sum_{j \ne i} \abs{m_{i,j}} (x_j^U(t) - x_j^L(t)),
\]  
or in matrix form,
\[
\dot{\x}^U(t) - \dot{\x}^L(t) = \tilde{\mbf{M}}(\x^U(t) - \x^L(t))
\]
where each off-diagonal component of $\tilde{\mbf{M}}$ equals the absolute value of the corresponding component of $\mbf{M}$, and the diagonals of $\tilde{\mbf{M}}$ and $\mbf{M}$ are equal.

If every eigenvalue of $\tilde{\mbf{M}}$ has negative real part, then 
$\x^U(t) - \x^L(t) \to \mbf{0}$ as $t \to +\infty$
(see for instance \cite[Thm.~3.5]{khalil}.
Gershgorin's circle theorem \cite[\S7.4]{strang} provides a way to bound the eigenvalues;
for any eigenvalue $\lambda$, there is some $i$ such that
\[
   \abs{ \lambda - \tilde{m}_{i,i} } \le \sum_{j \ne i} \abs{ \tilde{m}_{i,j} }.
\]
Since each $m_{i,i}$ is real, if we require that for all $i$
\[
\tilde{m}_{i,i} + \sum_{j \ne i} \abs{ \tilde{m}_{i,j} } < 0,
\]
then we can be sure that all eigenvalues of $\tilde{\mbf{M}}$ have negative real part.

Of course, we realize that these conditions simplify to 
${m}_{i,i} + \sum_{j \ne i} \abs{ {m}_{i,j} } < 0$ 
for all $i$.
Finally, recall the definition $\mbf{M} = \mbf{A} - \mbf{LC}$, 
let the rows of $\mbf{L}$ be $\l_i \tr$, and
let the columns of $\mbf{C}$ be $\mbf{C}_i$.
Then these conditions become
\[
{a}_{i,i} - \l_i\tr \mbf{C}_i + \sum_{j \ne i} \abs{ {a}_{i,j} - \l_i\tr \mbf{C}_j } < 0.
\]
\end{proof}

%Certainly more general results are possible by including bounded inputs in the system dynamics and appropriately weakening the claim;
%however, the present result serves as adequate motivation.

Related results are in the literature, like \cite[Thm.~1]{efimovEA16}.
%Again, such claims rely on $\mbf{A} - \mbf{LC}$ being Metzler.
Stability/asymptotic properties of the interval state estimate are often stated as requiring $\mbf{A} - \mbf{LC}$ to have eigenvalues with negative real part;
it is clear, again using Gershgorin's circle theorem, that conditions~\eqref{eq:GainConstraints}, if satisfied, imply exactly this.
Of course, it is \emph{not} the stability of the modified system 
$\dot{\x} = (\mbf{A} - \mbf{LC})\x + \mbf{L}\y$
that matters -- this is mostly a coincidence;
it is the stability of the dynamics defining $\x^U - \x^L$ that determines their asymptotic properties.

To turn the conditions in Theorem~\ref{thm:LinearStability} into an implementable numerical procedure, the next result states that a gain matrix satisfying the conditions may be found as the solution of a linear program (LP).

\begin{corollary}
\label{cor:GainCalculation}
Let the columns of $\mbf{C}$ be $\mbf{C}_i$.
Consider the LP
\begin{align}
\label{lp:gaincalc}
\min_{\mbf{L}, \mbf{B}, s}\; &s \\
\st
\notag   &\l_i\tr \mbf{C}_j - b_{i,j} \le  a_{i,j}, \quad \forall i, \forall j: j\ne i, \\
\notag  -&\l_i\tr \mbf{C}_j - b_{i,j} \le -a_{i,j}, \quad \forall i, \forall j: j\ne i, \\
\notag  -&\l_i\tr \mbf{C}_i + \smallsum_{j \ne i} b_{i,j} - s \le -a_{i,i}, \quad \forall i,
\end{align}
where the variables are $\mbf{L} = [\l_i\tr]$, $\mbf{B} = [b_{i,j}]$, and $s$.
Any $\mbf{L}$ which is a solution of this LP with optimal objective value $s^* < 0$ satisfies conditions~\eqref{eq:GainConstraints}.
\end{corollary} 
\begin{proof}
%First, note that the diagonal components of $\mbf{B}$ do not appear; they could be eliminated.
Note that the first two sets of constraints imply
$b_{i,j} \ge \abs{ a_{i,j} - \l_i\tr \mbf{C}_j}$, 
%or,
%$-b_{i,j} \le -\abs{ a_{i,j} - \l_i\tr \mbf{C}_j}$, 
for all $i \ne j$.
Then, with the last set of constraints, we get
\[
a_{i,i} - \l_i\tr \mbf{C}_i + \smallsum_{j \ne i} \abs{ a_{i,j} - \l_i\tr \mbf{C}_j} \le
a_{i,i} - \l_i\tr \mbf{C}_i + \smallsum_{j \ne i} b_{i,j} \le s,
\]
for each $i$.
Thus, if the optimal $s^*$ is negative, we certainly have that a corresponding solution $\mbf{L}$ satisfies conditions~\eqref{eq:GainConstraints}.
\end{proof}

As a practical note, we would add a (negative) lower bound on $s$ to the above LP in order to prevent the possibility that the LP is unbounded.
Further, the variables $b_{i,i}$ do not appear in the constraints or objective and could be eliminated;
a decent numerical LP solver will likely identify this automatically and eliminate them in a presolve step.

Other results for calculating the gain matrix based on the solution of an LP have been proposed;
see for instance \cite{chebotarevEA15}.
Again, these results rely on something like $\mbf{A} - \mbf{LC}$ being Metzler, although this extra structure permits more detailed claims about the input-output gains of the system.

%In the following examples, we will face nonlinear systems and construct the gain matrix via LP~\eqref{lp:gaincalc} using a linearized system;
%i.e. we will let $\mbf{A}$ equal the Jacobian matrix $\frac{\partial \f}{\partial \x}(\x^{ref})$ evaluated at some reference point.

%%%%%%%%%%%%%%%%%%%%%%%%%%%%%%%%%%%%%%%%%%%%%%%%%%%%%%%%%%%%%%%%%%%%%%%%%%%%%%%%
%%%%%%%%%%%%%%%%%%%%%%%%%%%%%%%%%%%%%%%%%%%%%%%%%%%%%%%%%%%%%%%%%%%%%%%%%%%%%%%%
\section{Examples}
\label{sec:examples}

We evaluate the performance of the proposed estimation method on some examples.
At the heart of the method is the solution of the IVP in ODEs in Theorem~\ref{thm:method}.
This initial value problem is solved with a C/C++ code employing the CVODE component of the SUNDIALS suite \cite{hindmarshEA05}.
Specifically, the numerical method uses the implementation of the Backwards Differentiation Formulae (BDF) in CVODE, using a Newton iteration for the corrector, with relative and absolute integration tolerances equal to $10^{-9}$.
The implementation of interval arithmetic in MC++ \cite{mc++} produces the inclusion function in Definition~\ref{defn:inclusionfun}.
%
%All numerical studies were performed on a Dell XPS laptop running 64-bit Ubuntu Linux, with an Intel Core i5-7200U CPU (2.50GHz) and 7.7 GB RAM.
%Computational times are for GCC with the -O3 optimization flag.

The proposed method will be referred to as GMAC, indicating that it incorporates 
\textbf{G}ained \textbf{M}easurement \textbf{A}nd \textbf{C}onstraint 
information when constructing the state estimate.
The No Constraints and No Measurement methods will be compared (recall discussion following Theorem~\ref{thm:method}).

%%%%%%%%%%%%%%%%%%%%%%%%%%%%%%%%%%%%%%%%%%%%%%%%%%%%%%%%%%%%%%%%%%%%%%%%%%%%%%%%
\subsection{Bioreactor}
\label{sec:bioreactor}

We consider an example involving a bioreactor originally from \cite{moisanEA05,meslemEA10}.
The dynamic model describes the evolution in time of the concentrations of biomass and feed substrate.
The dynamic equations on the time domain $T = [0,20]$ (day) are
\begin{align}
\notag
\dot{x}(t) &= \left(\mu_0(t) \frac{s(t)}{s(t) + k_s + s(t)^2/k_i} - \alpha D(t) \right) x(t), & x(0) \in [0,10] \text{ (mmol/L)}, \\
\notag
\dot{s}(t) &= -k \mu_0(t) x(t) \frac{s(t)}{s(t) + k_s + s(t)^2/k_i} + D(t)(s_{in}(t) - s(t)), & s(0) \in [0,100] \text{ (mmol/L)},
\end{align}
where $x(t)$ and $s(t)$ are the biomass and substrate concentrations, respectively, at time $t$, and the known parameters are 
\begin{align*}
&(k_s, k_i) = (9.28, 256) \text{ (mmol/L)}, \\
&(k, \alpha) = (42.14, 0.5), \\
&D(t) = 
\begin{cases}
2     \text{ (day}^{-1}), & \text{if } t \in [0,5],\\
0.5   \text{ (day}^{-1}), & \text{if } t \in (5, 10], \\
1.067 \text{ (day}^{-1}), & \text{if } t \in (10, 20].
\end{cases}
\end{align*}
Meanwhile, the unknown parameters/disturbances are $(\mu_0,s_{in})$ which satisfy for all $t \in T$
\begin{align*}
\mu_0(t)   &\in [0.703, 0.777] \text{ (day}^{-1}), \\
s_{in}(t)  &\in [0.95 \hat{s}_{in}(t), 1.05 \hat{s}_{in}(t)] \text{ (mmol/L)},
\end{align*}
where
$\hat{s}_{in} : t \mapsto 50 + 15 \cos(t/5)$
(so take $U : t \mapsto  [0.703, 0.777] \times [0.95 \hat{s}_{in}(t), 1.05 \hat{s}_{in}(t)]$).

Continuous measurements of $x(t)$, the biomass concentration, are available 
(so $\mbf{C} = [1\; 0]$).
The error%
\footnote{We note that the error in the \emph{initial} measurements of the biomass concentration are much larger than the subsequent measurements.
While in practice it seems likely that the initial measurements should be known at least as accurately as any subsequent online measurements, for consistency we follow this example as it appears in \cite{moisanEA05,meslemEA10}.
}
in these measurements satifies $v(t) \in [-0.25,0.25]$.
For simulation purposes, we obtain measurements from a numerical solution with 
$\mu_0(t) = 0.74$ for all $t$,
$s_{in}(t) = \hat{s}_{in}(t)$ for all $t$, 
$v(t) = 0$ for all $t$, and
initial conditions $x(0) = 5$ and $s(0) = 40$.
These measurements are obtained at 500 equally spaced time points from the interval $T$, and then made continuous by linearly interpolating between them.

In \cite{moisanEA05}, a ``bundle'' of interval observers are constructed, essentially by running independent estimates with different gain matrices $\mbf{L}$.
We choose one of these gain matrices $\mbf{L} = [2\; 0]\tr$.

Figure~\ref{fig:bioreactor} summarizes the results;
the upper and lower bounds of the interval estimate for each state is plotted versus time.
%We see trends similar to those in \S\ref{sec:efimov};
The proposed GMAC method is strictly tighter than either the No Measurements or No Constraints methods for at least one of the states.
The interval estimates at the final time are given in Table~\ref{tab:bioreactor}.

Compared to the bundle of observers from \cite{moisanEA05}, the estimates from GMAC are of comparable quality.
As mentioned, however, we achieve our results with only one observer gain matrix.
The same example is considered in \cite{meslemEA10}, and again we achieve similar state estimates.
The work in  \cite{meslemEA10} is based on a hybrid automata approach, and some extra analysis to determine the switching conditions is required.
Meanwhile, we rely on automatic construction of inclusion functions through implementations of interval arithmetic and little extra analysis is required.

\begin{figure}
\centering
\includegraphics{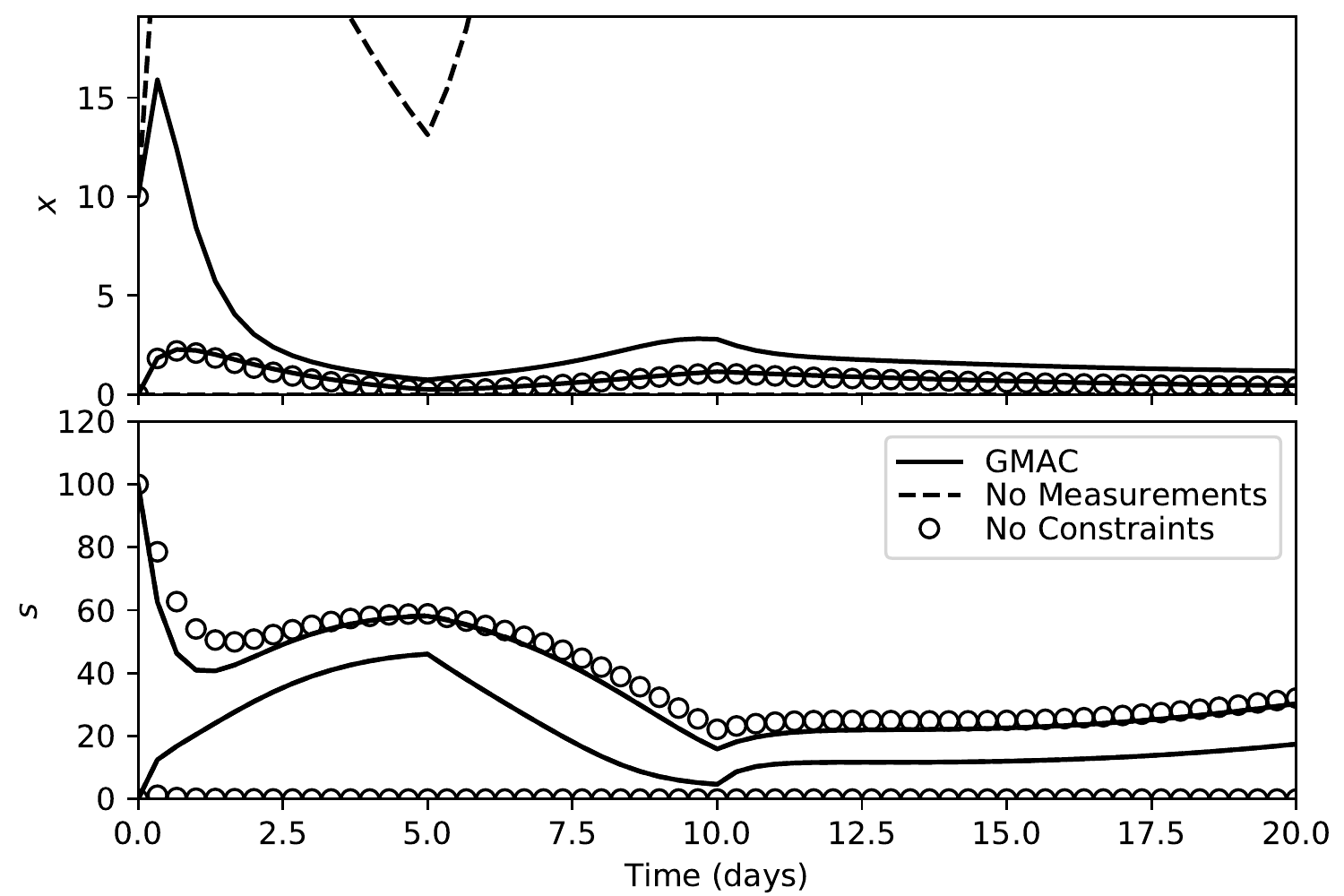}
\caption{State estimates for the bioreactor system from \S\ref{sec:bioreactor}:
biomass (top) and substrate (bottom).
The upper bounds for biomass, $x$, from the No Measurements and No Constraints methods are mostly off the bounds of the figure.
The estimates from the proposed GMAC and No Measurements methods are nearly indistinguishable for the substrate, $s$.}
\label{fig:bioreactor}
\end{figure}

\begin{table}
\caption{State estimates from \S\ref{sec:bioreactor} for various methods at final time: 
$(x(20),s(20)) \in [x^L, x^U] \times [s^L,s^U]$}
\label{tab:bioreactor}
\centering
\begin{tabular}{r | c c}
method            & $[x^L,x^U]$     & $[s^L,s^U]$ \\
\hline
GMAC              & $[0.449,1.19]$    & $[17.4, 30.3]$ \\
No Measurements   & $[0    ,10400]$   & $[17.4, 30.3]$ \\
No Constraints    & $[0.398,314000]$  & $[0   , 32.1]$
\end{tabular}
\end{table}

%%%%%%%%%%%%%%%%%%%%%%%%%%%%%%%%%%%%%%%%%%%%%%%%%%%%%%%%%%%%%%%%%%%%%%%%%%%%%%%%
\subsection{Linearized system}
\label{sec:efimov}

This example comes from Example 1 of \cite{efimovEA16}.
We have a three state system with nonlinear dynamics;
$\x$ satisfies
\begin{equation}
\notag
\dot\x(t) = 
\begin{bmatrix}
2 & 0 & 0\\
1 & -4 & \sqrt{3} \\
-1 & -\sqrt{3} & -4
\end{bmatrix}
\x(t) + 
\begin{bmatrix}
-2 u_1(t) x_1(t) x_2(t) \beta(t) \\
0 \\
   u_2(t) x_1(t) x_2(t) \beta(t) \\
\end{bmatrix},
\quad \alev t,
\end{equation}
where
\[
\beta : t \mapsto 1 + \sin(2t).
\]
The time interval of interest is $T = [0,5]$,
with the state known at $t = 0$: $\x(0) = (1,1,0)$.
The disturbances satisfy $(u_1(t),u_2(t)) \in [4.48,6.12] \times [3.2,3.6]$.
Measurements of the first state are available, so that $\mbf{C} = [1\; 0\; 0]$.
The error in these measurements satisfies $v(t) \in [-0.1, 0.1]$.
For simulation purposes, we obtain measurements from a numerical solution of the system with 
$u_1(t) = 5.3$ for all $t$,
$u_2(t) = 3.4$ for all $t$, and 
$v(t) = 0.1\sin(10t)$.
These measurements are obtained at 500 equally spaced time points from the interval $T$, and then made continuous by linearly interpolating between them.

Although this is a nonlinear system, we attempt to design an observer gain matrix based on the linear part.
The original analysis in \cite[Example 1]{efimovEA16} suggests
\[
\mbf{L} = \mbf{L}_1 \equiv [3\;\; 0\;\; 0]\tr.
\]
However, solving LP~\eqref{lp:gaincalc} yields 
\[
\mbf{L} = \mbf{L}_2 \equiv [4.27\;\; 1\;\; {-}1]\tr,
\]
and a negative optimal objective value is obtained.
We will calculate state estimates with both separately.
Refer to these as ``gain 1'' and ``gain 2,'' respectively.
Recall that the No Measurements method is defined by using $\mbf{L} = \mbf{0}$, and so it remains the same.

Figure~\ref{fig:efimov} summarizes the results;
the upper and lower bounds of the interval estimate for each state is plotted versus time.
For this example, omitting constraints produces a much more conservative interval state estimate.
In fact, the bounds produced by the No Constraints method with gain 1 diverge and the numerical integration procedure prematurely terminates shortly after $t=3.5$
(specifically, the corrector iteration of the BDF in CVODE fails).
Using gain 2, the No Constraints method performs much better and can produce state estimates for the entire time interval.

Meanwhile, the proposed GMAC method (with either gain- the results are largely the same) more consistently produces a tight interval estimate for each state.
The No Measurements method is comparable.
Table~\ref{tab:efimov} lists specific values at the final time.
%the interval estimate at $t=5$ from the proposed method is 
%\[
%\x(6) \in [3.196, 4.381] \times [0.465,0.588] \times [-0.750,-0.518],
%\]
%while from the ``no measurements'' method it is
%\[
%\x(6) \in [0.003, 609]  \times [0.465,0.588] \times [-0.750,-0.518].
%\]
The main difference is that the estimate of the first state from GMAC is much tighter.
This is somewhat of a moot point, however, as bounded-error measurements of the first state are directly available;
for our simulated measurement values, this implies the interval estimate 
$x_1(5) \in [0.692, 0.892]$.

\begin{figure}
\centering
\includegraphics{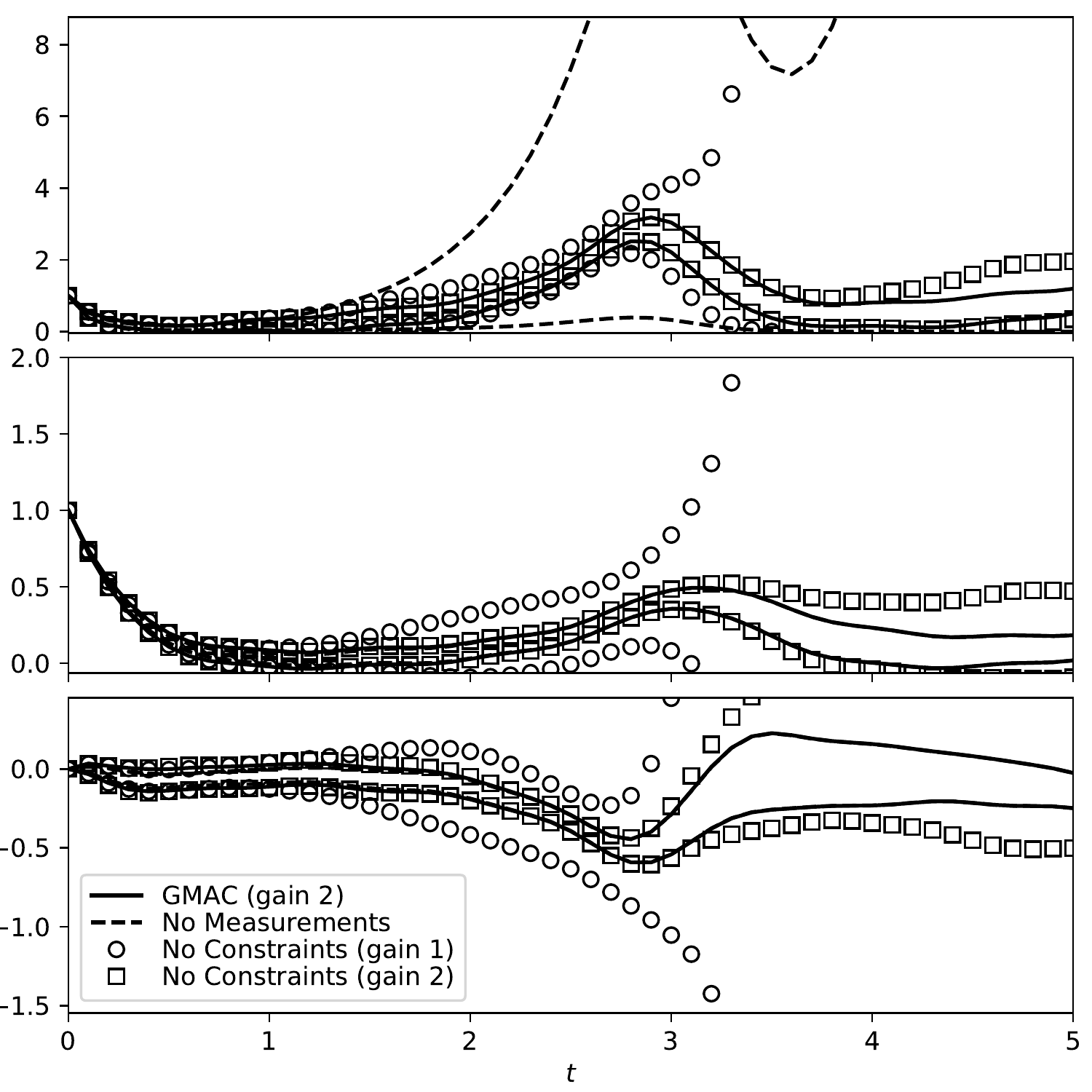}
\caption{Interval state estimates over time for each state in the example from \S\ref{sec:efimov}:
$x_1$ (top), $x_2$ (middle), $x_3$ (bottom).
The estimates from the proposed GMAC method and No Measurements method are nearly indistinguishable for $x_2$ and $x_3$.
``Gain i'' refers to using a gain matrix $\mbf{L} = \mbf{L}_i$ (for $i$ equal to 1 or 2).
GMAC with gain 1 is nearly identical to GMAC with gain 2.}
\label{fig:efimov}
\end{figure}

\begin{table}
\caption{State estimates from \S\ref{sec:efimov} for various methods at final time: 
$\x(5) \in [\x^L, \x^U]$}
\label{tab:efimov}
\centering
\begin{tabular}{r | c c c}
method & $[x_1^L, x_1^U]$ & $[x_2^L, x_2^U]$ & $[x_3^L, x_3^U]$ \\
\hline
GMAC (gain 2)           & $[0.504,1.20]$   & $[0.0178,0.182]$  &$[-0.248,-0.0250]$ \\
No Measurements         & $[0.000852,113]$ & $[0.0179,0.182]$  &$[-0.248,-0.0251]$ \\
No Constraints (gain 2) & $[0.350,1.96]$   & $[-0.0919,0.472]$ &$[-0.502,0.564]$
%GMAC (gain 2)           & $[0.5035477158,1.197492289]$    & $[0.01777222486,0.1822055495]$ &$[-0.2483334569,-0.02498321735]$ \\
%No Measurements         & $[0.0008518976006,112.7474634]$ & $[0.01786563544,0.1820958554]$ &$[-0.2482006055,-0.02513638129]$ \\
%No Constraints (gain 2) & $[0.3503455152,1.958578756]$    & $[-0.09191131863,0.4716091161]$ &$[-0.5016665726,0.5640607014]$
\end{tabular}
\end{table}

%%%%%%%%%%%%%%%%%%%%%%%%%%%%%%%%%%%%%%%%%%%%%%%%%%%%%%%%%%%%%%%%%%%%%%%%%%%%%%%%
%%%%%%%%%%%%%%%%%%%%%%%%%%%%%%%%%%%%%%%%%%%%%%%%%%%%%%%%%%%%%%%%%%%%%%%%%%%%%%%%
\section{Conclusions}
\label{sec:conc}

This work has considered the problem of bounded-error state estimation for nonlinear systems.
Using recent work for computing the reachable sets of constrained dynamic systems, we proposed a method that uses measurements in a novel way to construct interval state estimates.
Numerical examples demonstrate that the method is effective in practice and can improve on existing approaches.
The proposed approach also benefits from its reliance on interval arithmetic and its implementation in numerical libraries, which reduces the amount of manual analysis of a specific problem one must perfom.
Finally, we proposed a new way of calculating the gain matrix and showed that it can produce asymptotically exact interval estimates in ideal cases.
%%%%%%%%%%%%%%%%%%%%%%%%%%%%%%%%%%%%%%%%%%%%%%%%%%%%%%%%%%%%%%%%%%%%%%%%%%%%%%%%
%%%%%%%%%%%%%%%%%%%%%%%%%%%%%%%%%%%%%%%%%%%%%%%%%%%%%%%%%%%%%%%%%%%%%%%%%%%%%%%%

%%%%%%%%%%%%%%%%%%%%%%%%%%%%%%%%%%%%%%%%%%%%%%%%%%%%%%%%%%%%%%%%%%%%%%%%%%%%%%%%
%%%%%%%%%%%%%%%%%%%%%%%%%%%%%%%%%%%%%%%%%%%%%%%%%%%%%%%%%%%%%%%%%%%%%%%%%%%%%%%%

\begin{thebibliography}{10}

\bibitem{bokanowskiEA10}
Olivier Bokanowski, Nicolas Forcadel, and Hasnaa Zidani.
\newblock {Reachability and minimal times for state constrained nonlinear
  problems without any controllability assumption}.
\newblock {\em SIAM Journal on Control and Optimization}, 48(7):4292--4316,
  2010.

\bibitem{mc++}
Beno\^{\i}t Chachuat.
\newblock {MC++: A Versatile Library for McCormick Relaxations and Taylor
  Models}.
\newblock \url{http://www.imperial.ac.uk/people/b.chachuat/research.html},
  2015.

\bibitem{chebotarevEA15}
Stanislav Chebotarev, Denis Efimov, Tarek Ra{\"\i}ssi, and Ali Zolghadri.
\newblock Interval observers for continuous-time lpv systems with l1/l2
  performance.
\newblock {\em Automatica}, 58:82--89, 2015.

\bibitem{efimovEA16}
Denis Efimov and Tarek Ra{\"\i}ssi.
\newblock Design of interval observers for uncertain dynamical systems.
\newblock {\em Automation and Remote Control}, 77(2):191--225, 2016.

\bibitem{efimovEA13}
Denis Efimov, Tarek Ra\"{\i}ssi, Stanislav Chebotarev, and Ali Zolghadri.
\newblock {Interval state observer for nonlinear time varying systems}.
\newblock {\em Automatica}, 49(1):200--205, 2013.

\bibitem{harrison77}
Gary~W. Harrison.
\newblock {Dynamic models with uncertain parameters}.
\newblock In X~.J.~R. Avula, editor, {\em Proceedings of the First
  International Conference on Mathematical Modeling}, volume~1, pages 295--304.
  University of Missouri Rolla, 1977.

\bibitem{harwoodEA16}
Stuart~M. Harwood and Paul~I. Barton.
\newblock Efficient polyhedral enclosures for the reachable set of nonlinear
  control systems.
\newblock {\em Mathematics of Control, Signals, and Systems}, 28(1):8, 2016.

\bibitem{harwoodEA17}
Stuart~M. Harwood and Paul~I. Barton.
\newblock Affine relaxations for the solutions of constrained parametric
  ordinary differential equations.
\newblock {\em Optimal Control Applications and Methods}, 2017.

\bibitem{harwoodEA15}
Stuart~M. Harwood, Joseph~K Scott, and Paul~I. Barton.
\newblock Bounds on reachable sets using ordinary differential equations with
  linear programs embedded.
\newblock {\em IMA Journal of Mathematical Control and Information},
  33(2):519--541, 2016.

\bibitem{hindmarshEA05}
Alan~C. Hindmarsh, Peter~N. Brown, Keith~E. Grant, Steven~L. Lee, Radu Serban,
  Dan~E. Shumaker, and Carol~S. Woodward.
\newblock {SUNDIALS: Suite of Nonlinear and Differential/Algebraic Equation
  Solvers}.
\newblock {\em ACM Transactions on Mathematical Software}, 31(3):363--396,
  2005.

\bibitem{jaulin02}
L.~Jaulin.
\newblock {Nonlinear bounded-error state estimation of continuous-time
  systems}.
\newblock {\em Automatica}, 38(6):1079--1082, 2002.

\bibitem{kalman60}
R.~E. Kalman.
\newblock {A New Approach to Linear Filtering and Prediction Problems}.
\newblock {\em Journal of Basic Engineering}, 82:35--45, 1960.

\bibitem{khalil}
Hassan~K. Khalil.
\newblock {\em {Nonlinear Systems}}.
\newblock Prentice-Hall, Upper Saddle River, NJ, second edition, 1996.

\bibitem{knuppel94}
O.~Kn{\"u}ppel.
\newblock {PROFIL}/{BIAS} --- {A} fast interval library.
\newblock {\em Computing}, 53(3):277--287, Sep 1994.

\bibitem{kurzhanskiEA06_optimization}
A.~B. Kurzhanski, I.~M. Mitchell, and P.~Varaiya.
\newblock {Optimization techniques for state-constrained control and obstacle
  problems}.
\newblock {\em Journal of Optimization Theory and Applications},
  128(3):499--521, 2006.

\bibitem{kurzhanskiEA06}
A.~B. Kurzhanski and P.~Varaiya.
\newblock {Ellipsoidal techniques for reachability under state constraints}.
\newblock {\em SIAM Journal on Control and Optimization}, 45(4):1369--1394,
  January 2006.

\bibitem{lambert}
J.~D. Lambert.
\newblock {\em {Numerical Methods for Ordinary Differential Systems: The
  Initial Value Problem}}.
\newblock John Wiley \& Sons, New York, 1991.

\bibitem{meslemEA11}
N.~Meslem and N.~Ramdani.
\newblock {Interval observer design based on nonlinear hybridization and
  practical stability analysis}.
\newblock {\em International Journal of Adaptive Control and Signal
  Processing}, 25:228--248, 2011.

\bibitem{meslemEA10}
N.~Meslem, N.~Ramdani, and Y.~Candau.
\newblock {Using hybrid automata for set-membership state estimation with
  uncertain nonlinear continuous-time systems}.
\newblock {\em Journal of Process Control}, 20:481--489, 2010.

\bibitem{moisanEA05}
M.~Moisan and O.~Bernard.
\newblock {Interval observers for non-montone systems. Application to
  bioprocess models}.
\newblock {\em 16th IFAC World Congress}, 16:43--48, 2005.

\bibitem{moisanEA07}
M.~Moisan and O.~Bernard.
\newblock {An interval observer for non-monotone systems: application to an
  industrial anaerobic digestion process}.
\newblock {\em 10th International IFAC Symposium on Computer Applications in
  Biotechnology}, 1:325--330, 2007.

\bibitem{moore_etal}
Ramon~E. Moore, R.~Baker Kearfott, and Michael~J. Cloud.
\newblock {\em {Introduction to Interval Analysis}}.
\newblock SIAM, Philadelphia, 2009.

\bibitem{rump99}
Siegfried~M. Rump.
\newblock {INTLAB} --- {INT}erval {LAB}oratory.
\newblock In Tibor Csendes, editor, {\em Developments in Reliable Computing},
  pages 77--104. Springer Netherlands, Dordrecht, 1999.

\bibitem{scottEA13}
Joseph~K. Scott and Paul~I. Barton.
\newblock {Bounds on the reachable sets of nonlinear control systems}.
\newblock {\em Automatica}, 49(1):93--100, 2013.

\bibitem{strang}
Gilbert Strang.
\newblock {\em {Linear Algebra and its Applications}}.
\newblock Thomson Brooks/Cole, fourth edition, 2006.

\bibitem{thabetEA14}
Rihab El~Houda Thabet, Tarek Ra\"{\i}ssi, Christophe Combastel, Denis Efimov,
  and Ali Zolghadri.
\newblock {An effective method to interval observer design for time-varying
  systems}.
\newblock {\em Automatica}, 50(10):2677--2684, 2014.

\bibitem{walter}
Wolfgang Walter.
\newblock {\em {Differential and Integral Inequalities}}.
\newblock Springer, New York, 1970.

\end{thebibliography}
\end{document}